\documentclass[10pt,reqno]{amsart}
\usepackage{amsopn,enumerate,url}

\theoremstyle{plain}
\newtheorem{theorem}{Theorem}[section]
\newtheorem{lemma}[theorem]{Lemma}
\newtheorem{corollary}[theorem]{Corollary}
\newtheorem{proposition}[theorem]{Proposition}

\theoremstyle{remark}
\newtheorem*{remark*}{Remark}

\newcommand{\FF}{\mathbb F}
\newcommand{\ZZ}{\mathbb Z}
\newcommand{\PP}{\mathbb P}

\DeclareMathOperator{\Jac}{Jac}
\DeclareMathOperator{\Tr}{Tr}
\DeclareMathOperator{\Gal}{Gal}
\DeclareMathOperator{\Aut}{Aut}
\DeclareMathOperator{\Diff}{Diff}
\DeclareMathOperator{\PGU}{PGU}
\DeclareMathOperator{\Sz}{Sz}
\DeclareMathOperator{\Ree}{Ree}

\newcommand{\rcf}{{\text{rcf}}}

\newcommand{\mc}[1]{\mathcal{#1}}
\newcommand{\mf}[1]{\mathfrak{#1}}

\subjclass[2010]{Primary 11G20; Secondary 14H25.}

\begin{document}

\title[New maximal curves as ray class fields over Deligne-Lusztig curves]{New maximal curves as ray class fields over Deligne-Lusztig curves}

\author{Dane C. Skabelund}
\address{
Department of Mathematics\\
University of Illinois\\
Urbana, IL 61801\\
U.S.A.}
\email{skabelu2@illinois.edu}

\begin{abstract}
  We construct new covers of the Suzuki and Ree curves which are maximal with respect to the Hasse-Weil bound over suitable finite fields.
  These covers are analogues of the Giulietti-Korchm\'aros curve, which covers the Hermitian curve and is maximal over a base field extension.
  We show that the maximality of these curves implies that of certain ray class field extensions of each of the Deligne-Lusztig curves.
  Moreover, we show that the Giulietti-Korchm\'aros curve is equal to the above-mentioned ray class field extension of the Hermitian curve.
\end{abstract}

\maketitle

\section{Introduction}

A smooth, geometrically irreducible, projective algebraic curve (henceforth \linebreak \emph{curve}) $X$ defined over a finite field $\FF_q$ is called maximal (over $\FF_q$) if it meets the Hasse-Weil upper bound for the number of $\FF_q$-rational points for its genus, that is, if
\[
  \#X(\FF_q) =  q + 1 + 2g q^{1/2} .
\]
This is equivalent to saying that the $L$-polynomial of $X$ over $\FF_q$ is of the form $L_X(T) = (1+q^{1/2} T)^{2g}$.
Interest in studying maximal curves was sparked in the 1980s by their application in the theory of algebraic geometry codes.
An interesting observation, often attributed to Serre (see \cite{Lachaud}, Prop.\ 6), is that if $X \to Y$ is a nonconstant map of curves defined over $\FF_q$, then there is divisibility of $L$-polynomials in the opposite direction.
This implies that any curve covered by a $\FF_q$-maximal curve is also $\FF_q$-maximal.
This observation motivates the question of determining which $\FF_q$-maximal curves are largest with respect to the partial order on the set of $\FF_q$-maximal curves induced by coverings of one curve by another.

It was shown by Ihara \cite{Ihara} that an $\FF_q$-maximal curve has genus at most $q_0(q_0-1)/2$, where $q = q_0^2$.
Since this bound is attained by the Hermitian curve $H = H_{q_0}$ defined by
\[
  y^{q_0} + y = x^{q_0+1},
\]
the curve $H$ is not covered by any other $\FF_q$-maximal curve.

Other examples of maximal curves include the Suzuki and Ree curves.
These curves are defined over finite fields $\FF_q$ of odd prime power order in characteristic 2 and 3, respectively.
Each has an optimal number of $\FF_q$-rational points for its genus, and becomes maximal over a suitable extension of the base field.
Together with the Hermitian curve they constitute the Deligne-Lusztig curves associated to the simple groups of type ${}^2A_2$, ${}^2B_2$, and ${}^2G_2$ \cite{DL} \cite{Hansen}.
It has been shown that the smallest examples of the Suzuki and Ree curves are not Galois covered by the Hermitian curve maximal over the same field \cite{MZ}.
Each of the Deligne-Lusztig curves can be described as a composite of Artin-Schreier extensions over $\PP^1$ with very similar ramification structure, and Lauter has shown \cite{Lauter} that their function fields admit a unified description as ray class fields over $\FF_q(x)$ with conductor supported at a single rational point.

In 2008, Giulietti and Korchm\'aros \cite{GK} discovered a new family of curves maximal over $\FF_{q^3}$ for $q = q_0^2$ a square, each of which is not covered by the Hermitian curve $H_{q_0^3}$.
We denote this ``GK-curve'' by $\widetilde H = \widetilde H(q)$.
The curve $\widetilde H$ has genus $\frac 12 q(q+q_0-1)(q_0-1)$ and its function field may be defined by the equations
\[
  y^{q_0}+y = x^{q_0+1},
  \qquad t^m = x^q - x,
\]
where $m = q - q_0 + 1$.
These equations describe the curve $\widetilde H$ as the fiber product of the cover $H \to \PP^1_x$ with a Kummer extension of $\PP^1_x$ of degree $m$.

The results in this paper are motivated by the observation that the abelian cover $\widetilde H \to H$ is tamely ramified at each $\FF_q$-rational point of $H$, and that every other $\FF_{q^3}$-rational point of $H$ splits completely in $\widetilde H$.
Thus, the function field $\FF_{q^3}(\widetilde H)$ is contained in the maximal abelian extension of $\FF_{q^3}(H)$ of conductor $\mf m = \sum_{P \in H(\FF_q)} P$ in which each point of $\Sigma = H(\FF_{q^3}) \setminus H(\FF_q)$ splits.
The corresponding curve $H_\rcf$ is also maximal over $\FF_{q^3}$ and is in fact equal to $\widetilde H$, as we  show in section \ref{sec RCF}.
We note that the splitting set $\Sigma$ consists of the second short orbit under the action of $\Aut(H) = \PGU(3,q_0)$, the first being the set of $\FF_q$-rational points.
This provides an intrinsic alternate description of the GK-curve as arising from the Hermitian curve in a natural way.
Moreover, it motivates our analogous contruction of maximal covers of the Suzuki and Ree curves.

We outline the remaining contents of the paper.
In section \ref{sec DL curves} we give background information about the Deligne-Lusztig curves.
In sections \ref{sec wtS} and \ref{sec wtR} we introduce cyclic covers $\widetilde S$ and $\widetilde R$ of the Suzuki and Ree curves analogous to the GK-curve, and show that these are maximal over a suitable base field.
In section \ref{sec RCF}, we consider for each Deligne-Lusztig curve $X$ a cover $X_\rcf \to \widetilde X$ obtained from $X$ via a ray class field construction, and show that the maximality of $\widetilde X$ implies that of $X_\rcf$.
We then show that the cover $X_\rcf \to X$ is cyclic of a prescribed form.
This allows us to show that $H_\rcf = \widetilde H$, and to verify computationally that $X_\rcf = \widetilde X$ for small values of $q$ when $X$ is one of the Suzuki or Ree curves.
We leave as an open problem whether $X_\rcf = \widetilde X$ in general, but give a bound on the degree of the cover $X_\rcf \to \widetilde X$.

The curves $X_\rcf$ can be thought of as arising via a two-step ray class field contruction over $\PP^1$.
First, apply the construction of Lauter allowing one $\FF_q$-rational point to ramify and splitting all other rational points.
Then extend the base field before taking another ray class field in which the $\FF_q$-rational points are allowed to ramify tamely, while all other points rational over the new base field are caused to split.

\section{The Deligne-Lusztig Curves}\label{sec DL curves}

In this section we collect some facts about the Deligne-Lusztig curves which will be used later.
In particular, we note that the Hermitian, Suzuki, and Ree curves have exactly two short orbits under the action of their full automorphism group, each consisting of all points of certain degrees.
These orbits form the sets of points which ramify and split in the covers constructed in later sections.
The results in this section should be well known; we include proofs of some statements for the sake of completeness.

A point of a curve $X$ is fixed by some automorphism of $X$ exactly if it is ramified in the quotient by the full automorphism group.
Thus, the Riemann-Hurwitz formula may be applied to the cover $X \to X/\Aut(X)$ to study these points.
For $G$ a finite subgroup of automorphisms of $X$, this formula may be written in the form
\[
 2g(X) - 2 = \#H\left( 2g(X/G)-2 + \sum_{\mf p \in X/G} \frac{d(\mf p)}{e(\mf p)} \deg \mf p \right) ,
\]
where $d(\mf p)$ and $e(\mf p)$ are the different exponent and ramification index of $\mf p$ in $X \to X/G$.

\subsection{The Hermitian Curve}

Let $q = q_0^2$ be a prime power. The Hermitian curve $H = H_{q_0}$ has an affine plane model defined by
\[
  y^{q_0}+y = x^{q_0+1} .
\]
It has genus $q_0(q_0-1)/2$ and is maximal over $\FF_q$, with $\#X(\FF_q)= q^{3/2}+1$.
Its automorphism group $\PGU(3,q_0)$ is of size $(q^{3/2}+1)q^{3/2}(q-1)$.

\begin{proposition}[\cite{GSX}, Proposition 2.2]
  The Hermitian curve $H_{q_0}$ has exactly two short orbits under the action of its full automorphism group.
  One is non-tame of size $q^{3/2}+1$, consisting of the $\FF_q$-rational points. The other is tame of size $\frac 13 q^{3/2}(q-1)(q_0+1)$, consisting of all points of degree 3.
\end{proposition}

\subsection{The Suzuki Curve}

For $s \geq 1$ and $q = 2q_0^2 = 2^{2s+1}$, the Suzuki curve $S/\FF_q$ has an affine model defined by
\[
  y^q + y = x^{q_0}(x^q+x)
\]
and has genus $q_0(q-1)$.
Its automorphism group $\Sz(q)$, which has size $(q^2+1)q^2(q-1)$, acts doubly transitively on the $q^2+1$
rational points.
The Suzuki curve is maximal over $\FF_{q^4}$.

\begin{proposition}\label{suz_orbits}
  The Suzuki curve has exactly two short orbits under the action of its full automorphism group. One is non-tame of size $q^2+1$, consisting of the $\FF_q$-rational points. The other is tame of size $\frac 14 q^2(q-1)(q+2q_0 + 1)$, consisting of all points of degree 4.
\end{proposition}

\begin{proof}
The automorphism group $G= \Sz(q)$ transitive on the $q^2+1$ points of $S(\FF_q)$ with point stabilizer of order $q^2(q-1)$.
Fix a rational point $\mf P_\infty$ on $S$ lying above a point $\mf p_\infty$ in $S/G \cong \PP^1$.
Then
\[
  \#G_\infty = e(\mf P_\infty|\mf p_\infty) = q^2(q-1) .
\]
Let $G_i$ denote the (lower) ramification groups at $\mf P_\infty$.
Then (\cite{HKT}, \S12.2)
\begin{align*}
 &\#G_1 = q^2, \\
 &\#G_2 = \cdots = \#G_{2q_0+1} = q, \\
 &\#G_{2q_0+2} = 1 .
\end{align*}
Therefore, the different exponent at $\mf p_\infty$ is
\begin{align*}
 d(\mf p_\infty)
&= ( q^2(q-1) - 1) + (q^2 - 1) + 2q_0 \cdot (q - 1) \\
&= q^3+2qq_0 - 2q_0-2 .
\end{align*}
From the Hurwitz formula, it follows that
\[
 \sum_{\mf p \neq \mf p_\infty} \frac{d(\mf p)}{e(\mf p)} \deg \mf p
= \frac{q-2q_0}{q-2q_0+1}
< 1 .
\]
Now $d(\mf p) \geq e(\mf p) - 1$, with equality if and only if $\mf p$ is tamely ramified.
Thus the inequality above implies that there is exactly one place $\mf p$ of $S/G$ other than $\mf p_\infty$ which is ramified in $S \to S/G$, and that $\deg \mf p = 1$.
Moreover, $\mf p$ is tamely ramified with $e(\mf p) = q-2q_0+1$.

Let $\mf P$ be a prime of $S$ lying over $\mf p$.
Then the inertia group $I = I(\mf P|\mf p)$ is cyclic of order $q-2q_0+1$.
There is a unique conjugacy class of cyclic subgroups of this order ($I$ is a Singer subgroup, see \cite{Huppert}).
The decomposition group $D$ of $\mf P$ has size $(\deg \mf P) [G:I]$.
Since $N_G(I)$ is the unique maximal subgroup containing $I$, $D \subset N_G(I)$ and $\deg \mf P$ divides $[N_G(I):I] = 4$.
But $\mf P$ does not have degree 1 because it is not conjugate to $\mf P_\infty$, and $S$ has no points of degree 2, so $\deg \mf P = 4$ and $D = N_G(I)$.
Thus, the point $\mf P$ has orbit of size $[G:D] = \frac 14 q^2(q-1)(q+2q_0 + 1)$.
Since this equal to the number of points on $S$ of degree 4, these points form a single orbit under the action of $G$.
\end{proof}

\subsection{The Ree Curve}

For $s \geq 1$ and $q = 3q_0^2 = 3^{2s+1}$, the Ree curve $R/\FF_q$ may be defined by the affine equations
\begin{align*}
  y^q - y &= x^{q_0}(x^q-x), \\
  z^q - z &= x^{2q_0}(x^q-x) .
\end{align*}
The curve $R$ has $q^3+1$ points rational over $\FF_q$, genus $\frac 32 q_0(q-1)(q+q_0+1)$, and automorphism group $\Ree(q)$ of size $(q^3+1)q^3(q-1)$.
The Ree curve is maximal over $\FF_{q^6}$.

\begin{proposition}
  The Ree curve has exactly two short orbits under the action of its full automorphism group.
  One is non-tame of size $q^3+1$, consisting of the $\FF_q$-rational points.
  The other is tame of size $\frac 16 q^3(q-1)(q+1)(q+3q_0+1)$, consisting of all points of degree 6.
\end{proposition}

\begin{proof}

The proof of this proposition is similar to that of Proposition \ref{suz_orbits}.
The only pieces of information needed are the sizes of the ramifications groups at a rational point of $R$ in $R \to R/\Aut(R)$, and a list of the maximal subgroups of $\Aut(R)$.
These may be found in \cite{HP} and on page 648 of \cite{HKT}, respectively. \qedhere

\end{proof}

\section{The Curve $\widetilde S$}\label{sec wtS}

In this section, we introduce a cover $\widetilde S \to S$ of the Suzuki curve which is analogous to the GK-curve $\widetilde H \to H$ and show that it is maximal over $\FF_{q^4}$.
Fix $s \geq 1$, let $q = 2q_0^2 = 2^{2s+1}$, and recall the definition of the Suzuki curve $S$ from section \ref{sec DL curves}.
Let $\widetilde S$ be a smooth model of the curve with function field described by
\begin{align*}
  y^q+y &= x^{q_0}(x^q+x), \\
  t^m &= x^q+x ,
\end{align*}
where $m = q-2q_0+1$.
The curve $\widetilde S$ may be described as the normalization of the fiber product of the covers $S \to \PP_x^1$ and $C_m \to \PP_x^1$, where $C_m$ is the curve described by the second equation above.

Let $F = \FF_q(x)$.
The function field $\FF_q(\widetilde S)$ is the composite of $\FF_q(C_m) = F(t)$ and $\FF_q(S) = F(y)$.
Each place of $F$ of degree 1 is ramified in $F(t)$, with ramification index $m$, and no other places are ramified.
Also, the place $\infty$ corresponding to $1/x$ is the only place ramified in $F(y)$, with ramification index $q$.
Therefore, the only places ramified in $F(t,y)/F(y)$ are the $q^2+1$ rational places, and each is tamely ramified with ramification index $m$.
Thus, the Hurwitz formula gives
\[
  g_{\widetilde S}
  = 1 + m(g_S-1) + \frac 12 (q^2+1)(m-1)
  = \frac 12 (q^3-2q^2+q) .
\]

\begin{theorem}\label{Smax}
  The curve $\widetilde S$ is maximal over $\FF_{q^4}$.
\end{theorem}

We prove this by means of the following ``Natural Embedding Theorem'' of Korchm\'aros and Torres.
\begin{theorem}[Korchm\'aros-Torres \cite{KT}] \label{NET}
  A smooth geometrically irreducible  projective curve is maximal over $\FF_{\ell^2}$ if and only if it admits an embedding as a curve of degree $\ell+1$ in a non-degenerate Hermitian variety defined over $\FF_{\ell^2}$.
\end{theorem}

Here a Hermitian variety means a projective hypersurface defined by the vanishing of
\[
  \sum_{0 \leq i,j\leq n} a_{ij} x_ix_j^\alpha ,
\]
where $\alpha$ is the involutive automorphism $w \mapsto w^\ell$ of $\FF_{\ell^2}$, and $A = [a_{ij}]$ is some ``Hermitian'' matrix satisfying $A^t = A^\alpha$.

Given an $\FF_{\ell^2}$-maximal curve $X$, there is a concrete construction described in \cite{KT} for an embedding as in Theorem \ref{NET}.
First take a basis $f_1,\ldots,f_m$ for the linear series $L((\ell+1)P_0)$, where $P_0$ is any rational point of $X$.
Then there is a unique point $(z_0:\cdots:z_m) \in \PP^m_{\FF_{\ell^2}(X)}$ satisfying
\[
  z_1^\ell f_1 + \cdots + z_m^\ell f_m = 0 .
\]
After a linear change of variables, and possibly taking a projection, the functions $z_i$ then give desired embedding.

Define functions $z = y^{2q_0} + x^{2q_0+1}$ and $w = xy^{2q_0}+z^{2q_0}$ on $\widetilde S$.
These satisfy the relations
\begin{equation}\label{hqh}
  z^q+z = x^{2q_0}(x^q+x),
  \qquad w^q+w = y^{2q_0}(x^q+x) .
\end{equation}
Moreover, if $\infty$ denotes the pole of the function $x$, then
\begin{align*}
-v_\infty(x) &= qm = q^2-2qq_0+q ,\\
-v_\infty(y) &= -(1+ \frac{1}{2q_0} )v_\infty(x) = q^2 - qq_0 + q_0 ,\\
-v_\infty(z) &= -(1+ \frac{1}{q_0} )v_\infty(x) = q^2-q+2q_0 ,\\
-v_\infty(w) &= -(1+\frac{1}{q_0} + \frac 1q )v_\infty(x) = q^2+1 ,\\
-v_\infty(t)  &= -\frac{q}{m} v_\infty(x) = q^2 .
\end{align*}

Since the semigroup generated by these numbers has genus $g(\widetilde S)$, the pole orders of these functions generate the Weierstrass semigroup of $\widetilde S$ at $\infty$.
In particular, the linear series $L((q^2+1)\infty)$ has a basis $\mc B$ of monomials in $1,x,y,z,w,t$.
This fact was useful for finding the equation of the particular Hermitian variety in which we wish to embed $\widetilde S$, since it allows a search for the functions $z_i$ mentioned above to be phrased as a linear algebra problem over a vector space with basis $b_i^{q^2k} b_j$ with $b_i,b_j \in \mc B$.
Such a computation performed in Magma \cite{Magma} resulted in the discovery of equation \eqref{HSuz} below.

\begin{lemma}
  Every automorphism of $S$ lifts to an automorphism of $\widetilde S$ defined over $\FF_{q^4}$.
\end{lemma}

\begin{proof}
  The group $\Aut \FF_q(S)$ is generated by the stabilizer of the point $\infty$, which consists of automorphisms $\psi_{abc}$ taking
  \begin{align*}
    x &\mapsto ax+b \\
    y &\mapsto a^{q_0+1}y + b^{q_0}x + c ,
  \end{align*}
  for $a \in \FF_q^\times$ and $b,c \in \FF_q$,
  along with an involution $\phi$ defined by $\phi(x) = z/w$ and $\phi(y) = y/w$, which swaps $\infty$ with another rational point (see \cite{HS} and \cite{Henn}).

To extend $\psi = \psi_{abc}$ to $\widetilde S$, we need
\[
  \psi(t)^m
  = \psi(x)^q - \psi(x)
  = a(x^q+x) .
\]
Fix a generator $\alpha$ of $\FF_q^\times$ and an $m$th root $\beta$ of $\alpha$ which is contained in $\FF_{q^4}$ since $m$ divides $q^4-1$.
Then we may take $\psi(t) = a^{1/m}t$, where $a^{1/m}$ is chosen consistently with the choice of $\alpha$ and $\beta$.

The automorphism $\phi$ may be lifted to an automorphism of $\widetilde S$ by $\phi(t) = t/w$.
Indeed, $\phi$ so defined satisfies
\[
  \phi(t)^m = \phi(x)^q + \phi(x) = (z/w)^q + z/w .
\]
To verify this, first multiply the desired equality by $w^{q+1}$ and use \eqref{hqh} to obtain
\begin{align*}
  t^m w^{2q_0}
  &= z^qw + zw^q \\
  &= w(z^q+z) + z(w^q+w) \\
  &= wx^{2q_0}(x^q+x) + zy^{2q_0}(x^q+x) .
\end{align*}
Thus the desired equation is equivalent to $w^{2q_0} = wx^{2q_0} + zy^{2q_0}$, which may now be verified by using the definitions of $z$ and $w$.
\end{proof}

\begin{lemma}
  The map $\pi = (1:x:t:z:w)$ defines a smooth embedding of the curve $\widetilde S$ in $\PP^4$.
\end{lemma}

\begin{proof}
Let $X_0,\ldots,X_4$ be homogeneous coordinates on $\PP^4$.
We first check that $\pi(\widetilde S)$ has no singular points on the affine piece $X_0 = 1$.
Here the equations
\begin{align*}
  z^q+z &= x^{2q_0}(x^q+x)\\
  w^q+w &= (x^{2q_0+1}+z)(x^q+x) \\
  t^m &= x^q+x
\end{align*}
give a matrix of derivatives of rank 3.
It remains to check that $\pi(\widetilde S)$ has no singular points on the hyperplane $X_0 = 0$.
Since the function defining $\pi$ are in $L((q^2+1)\infty)$ and
\[
v_\infty(w) = -(q^2+1) < v_\infty(f)
\]
for $f \in \{1,x,z,t\}$, the only point in $\pi(\widetilde S) \cap Z(X_0)$ is $P_\infty := \pi(\infty) = (0:0:0:0:1)$.

Let $\phi$ be the automorphism of $\widetilde S$ mentioned in the previous lemma, which acts on the image $\pi(\widetilde S) \subset \PP^4$ via the permutation $(04)(13)$ of homogeneous coordinates on $\PP^4$.
Since the point $P_0 = (1:0:0:0:0) \in \pi(\widetilde S)$ is nonsingular, so is the point $\phi(P_0) = P_\infty$.
\end{proof}

\begin{proof}[Proof of Theorem \ref{Smax}]
We claim that
\begin{equation}\label{HSuz}
  w^{q^2} + w + z^{q^2}x + x^{q^2}z = t^{q^2+1} ,
\end{equation}
so that the image of the map $\pi = (1:x:t:z:w)$ lies on the Hermitian hypersurface
\[
  X_0X_4^{q^2} + X_0^{q^2}X_4 + X_1^{q^2}X_3+X_1X_3^{q^2} = X_2^{q^2+1}
\]
in $\PP^4$.
By Theorem \ref{NET} this will complete the proof of the theorem.
Writing $f = x^q+x$ for convenience, we use \eqref{hqh} to rewrite the terms on the left hand side of the desired equation as
\begin{align*}
  w^{q^2} + w
  &= (y^{2q_0}f)^q + y^{2q_0}f \\
  &= (y^{2q_0}+x^{2qq_0+q}+x^{q+2q_0})f^q + y^{2q_0}f
\end{align*}
and
\begin{align*}
  z^{q^2} x + x^{q^2}z
  &= (z^{q^2}+z)x + (x^{q^2}+x)z \\
  &= ((x^{2q_0}f)^q+x^{2q_0}f)x + (f^q+f)(x^{2q_0+1}+y^{2q_0}) \\
  &= ( x^{2qq_0+1} + x^{2q_0+1}+y^{2q_0} )f^q + y^{2q_0}f .
\end{align*}
Adding these gives
\[
  (x^{2qq_0+q}+x^{q+2q_0} + x^{2qq_0+1} + x^{2q_0+1} )f^q
  = f^{q+2q_0+1}
  = t^{q^2+1},
\]
and so the claim is proven.
\end{proof}

\section{The Curve $\widetilde R$}\label{sec wtR}

Fix $s \geq 1$, and let $q = 3q_0^2 = 3^{2s+1}$, and recall the definition of the Ree curve $R$ from section \ref{sec DL curves}.
In this section, we construct a cover $\widetilde R \to R$ which is also maximal over $\FF_{q^6}$.
Let $\widetilde R$ be a smooth model of the curve with function field described by
\begin{align*}
  y^q-y &= x^{q_0}(x^q-x), \\
  z^q-z &= x^{2q_0}(x^q-x), \\
  t^m &= x^q-x ,
\end{align*}
where $m = q-3q_0+1$.
The curve $\widetilde R$ may be described as the normalization of the fiber product of of the covers $R \to \PP_x^1$ and $C_m \to \PP_x^1$, where $C_m$ is the curve described by the third equation above.

Let $F = \FF_q(x)$.
The function field $\FF_q(\widetilde R)$ is the composite of $\FF_q(C_m) = F(t)$ and $\FF_q(S) = F(y,z)$.
Each place of $F$ of degree 1 is ramified in $F(t)$ with ramification index $m$, and no other places are ramified.
Also, the place $\infty$ corresponding to $1/x$ is the only place ramified in $F(y,z)$, with ramification index $q^2$.
Therefore, the only places ramified in $F(t,y,z)/F(y,z)$ are the $q^3+1$ rational places, and each is tamely ramified with ramification index $m$.
Thus, the Hurwitz formula gives
\begin{align*}
  g_{\widetilde R}
  = 1 + m(g_R-1) + \frac 12 (q^3+1)(m-1) 
  = \frac 12 (q^4-2q^3+q) .
%  = \frac 12 (q^4 + 6q^3q_0 + 2q^3 - 2q^2 - 6qq_0  - 3q + 2).
\end{align*}

\begin{theorem}\label{Rmax}
  The curve $\widetilde R$ is maximal over $\FF_{q^6}$.
\end{theorem}

We proceed as with the curve $\widetilde S$ by embedding $\widetilde R$ in a Hermitian hypersurface.
In \cite{Ped}, Pedersen defines ten functions $w_1,\ldots,w_{10}$ on the Ree curve by
\begin{align*}
w_1 &= x^{3q_0+1} - y^{3q_0}    &w_6 &= v^{3q_0} - w_2^{3q_0} + x w_4^{3q_0}\\
w_2 &= xy^{3q_0} - z^{3q_0} &w_7 &= w_2+v\\
w_3 &= xz^{3q_0} - w_1^{3q_0} &w_8 &= w_5^{3q_0}+xw_7^{3q_0}
\addtocounter{equation}{1}\tag{\theequation} \label{wdef} \\
w_4 &= xw_2^{q_0} - yw_1^{q_0} &w_9 &= w_4w_2^{q_0}-yw_6^{q_0}\\
v &= xw_3^{q_0} - zw_1^{q_0} & w_{10} &= zw_6^{q_0} - w_3^{q_0}w_4\\
w_5 &= yw_3^{q_0} - zw_1^{q_0}
\end{align*}
From the appendix of \cite{Ped}, these satisfy
\begin{align*}
  w_1^q-w_1 &= x^{3q_0}(x^q-x)     & w_4^q-w_4 &= (w_2-xw_1)^{q_0}(x^q-x) \\
  w_2^q-w_2 &= y^{3q_0}(x^q-x)     & w_6^q-w_6 &= w_4^{3q_0}(x^q-x)
  \addtocounter{equation}{1}\tag{\theequation} \label{wqw} \\
  w_3^q-w_3 &= z^{3q_0}(x^q-x)     & w_8^q-w_8 &= w_7^{3q_0}(x^q-x) .
\end{align*}
Furthermore, if $\infty$ denotes the unique pole of $x$ in $\widetilde R$, then it follows from the above that
\begin{align*}
-v_\infty(x)   &= q^3-3q^2q_0+q^2 ,
    & -v_\infty(w_6) &= q^3-q+3q_0 , \\
-v_\infty(w_1) &= q^3-2q^2+3qq_0,
    & -v_\infty(w_8) &= q^3+1,      \\
-v_\infty(w_2) &= q^3-q^2+q ,
    & -v_\infty(t)   &= q^3 . \\
-v_\infty(w_3) &= q^3-3qq_0+2q , &
\end{align*}
Duursma and Eid show in \cite{DE} that the 14 functions $1,x,y,z,w_1,\ldots,w_{10}$ define a smooth embedding of the Ree curve in $\PP^{13}$ and give a system of 105 equations describing this image, a few of which are referred to in this section.

\begin{lemma}
  Every automorphism of $R$ lifts to an automorphism of $\widetilde R$ defined over $\FF_{q^6}$.
\end{lemma}

\begin{proof}
We use the concrete description of $G = \Aut \FF_q(R)$ found in \cite{Ped}.
The group $G$ is generated by the stabilizer $G_\infty$ of the point $\infty$ and an involution $\phi$ which swaps $\infty$ with another rational point.
The stabilizer $G_\infty$ consists of automorphisms $\psi_{abcd}$ taking
\begin{align*}
  x &\mapsto ax+b\\
  y &\mapsto a^{q_0+1}y + ab^{q_0}x + c\\
  z &\mapsto a^{2q_0+1}z - a^{q_0+1}b^{q_0}y + ab^{2q_0}x + d,
\end{align*}
for $a \in \FF_q^\times$ and $b,c,d \in \FF_q$.
To extend $\psi = \psi_{abcd}$ to $\widetilde R$, we need
\[
  \psi(t)^m
  = \psi(x)^q - \psi(x)
  = a(x^q-x) .
\]
Fix a generator $\alpha$ of $\FF_q^\times$ and an $m$th root $\beta$ of $\alpha$ which is contained in $\FF_{q^6}$ since $m$ divides $q^6-1$.
Then we may take $\psi(t) = a^{1/m}t$, where $a^{1/m}$ is chosen consistently with the choice of $\alpha$ and $\beta$.

The involution $\phi \in \Aut \FF_q(R)$ mentioned above sends
\[
    x \mapsto w_6/w_8, \qquad
    y \mapsto w_{10}/w_8, \qquad
    z \mapsto w_9/w_8 .
\]
We claim that $\phi$ extends to $\widetilde R$ via $\phi(t) = t/w_8$.
To verify this, we show that
\[
  \phi(t)^m
  = \phi(x)^q - \phi(x)
  = (w_6/w_8)^q - w_6/w_8 .
\]
Upon multiplying through by $w_8^{q+1}$, using \eqref{wqw}, and then dividing by $x^q-x$, this is seen to be equivalent to
\[
  w_8^{3q_0} = w_8w_4^{3q_0} - w_6w_7^{3q_0} .
\]
But this is one of the equations appearing in Lemma 4.3 of \cite{DE}.
\end{proof}

\begin{lemma}
  The map $\pi = (1:x:w_1:w_2:t:w_3:w_6:w_8)$ defines a smooth embedding of the curve $\widetilde R$ in $\PP^7$.
\end{lemma}

\begin{proof}
Let $X_0,\ldots,X_7$ be homogeneous coordinates on $\PP^7$.
We first check that $\pi(\widetilde R)$ has no singular points on the affine piece $X_0 = 1$.
Here we have, from \eqref{wdef} and \eqref{wqw},
\begin{align*}
 w_1^q - w_1 &= x^{3q_0}(x^q-x) \\
 w_2^q - w_2 &= (x^{3q_0+1}-w_1)(x^q-x)\\
 w_3^q - w_3 &= (x^{3q_0+2}-xw_1 - w_2 )(x^q-x)  \\
 w_6^q - w_6 &= (x^{3q_0}w_2^q - w_1^qx^{3q_0+1} + w_1^{q+1} )(x^q-x)  \\
 w_8^q - w_8 &= ( w_2^{3q_0} + x^{3q_0}w_3^q - x^{6q_0+1}w_1^q + x^{3q_0}w_1^{q+1} + w_2w_1^q)(x^q-x)  .
\end{align*}
These equations, along with $t^m = x^q-x$, give a matrix of derivatives of rank 6.
It remains to show that $\pi(\widetilde R)$ has no singular points lying on the hyperplane $X_0 = 0$.
Since each of the functions defining $\pi$ are in $L((q^3+1)\infty)$, and
\[
  v_\infty(w_8) = -(q^3+1) < v_\infty(f)
\]
for $f \in \{ 1,x,w_1,w_2,w_3,w_6,t\}$, the only point of $\pi(\widetilde R) \cap Z(X_0)$ is $P_\infty := \pi(\infty) = (0:0:0:0:0:0:0:1)$.

Let $\phi$ be the involution in $\Aut \FF_q(\widetilde R)$ defined in the previous lemma, which takes $x \mapsto w_6/w_8$ and $t \mapsto t/w_8$.
Since the automorphism $\phi$ also sends
\begin{align*}
  w_1 &\mapsto w_3/w_8,
  &w_2 &\mapsto w_2/w_8,
  &w_3 &\mapsto w_1/w_8,  \\
  w_6 &\mapsto x/w_8,
  &w_8 &\mapsto w_8/w_8 ,
\end{align*}
it acts on the image $\pi(\widetilde R) \subset \PP^7$ via the permutation $(07)(16)(25)$ of homogeneous coordinates on $\PP^7$.
Since the point $P_0 = (1:0:0:0:0:0:0) \in \pi(\widetilde R)$ is nonsingular, so is the point $\phi(P_0) = P_\infty$.
\end{proof}

\begin{proof}[Proof of Theorem \ref{Rmax}]
We show that
\[
  w_8^{q^3}+w_8
  + xw_6^{q^3} + x^{q^3}w_6
  + w_1w_3^{q^3} + w_1^{q^3}w_3
  + w_2^{q^3+1}
  = t^{q^3+1} ,
\]
so that the image of the map $\pi = (1:x:w_1:w_2:t:w_3:w_6:w_8)$ in $\PP^7$ lies on the Hermitian hypersurface
\[
  X_0X_7^{q^3} + X_0^{q^3}X_7
  + X_1X_6^{q^3} + X_1^{q^3}X_6
  + X_2X_5^{q^3} + X_2^{q^3}X_5
  + X_3^{q^3+1}
  = X_4^{q^3+1} .
\]
The desired result will then follow by Theorem \ref{NET}.
Since
\[
  t^{q^3+1} = (x^q-x)^{\frac{q^3+1}{m}} = (x^q-x)^{q^2+3qq_0+2q+3q_0+1} ,
\]
our verification may be done completely inside of the function field $\FF_q(x,y,z)$ of $R$.
Writing $f = x^q-x$ for convenience, we use \eqref{wqw} to rewrite the terms on the left hand side of the desired equation as
\begin{align*}
  w_8^{q^3} + w_8
  &= (w_7^{3q_0})^{q^2}f^{q^2} + (w_7^{3q_0})^{q}f^q + w_7^{3q_0}f - w_8, \\
%%%
  xw_6^{q^3} + x^{q^3}w_6
  &= (w_6^{q^3}-w_6)x + (x^{q^3}-x)w_6 - xw_6 \\
  &= ( (w_4^{3q_0})^{q^2}x + w_6 )f^{q^2} \\
   &\qquad + ( (w_4^{3q_0})^{q}x + w_6 )f^q
   + ( w_4^{3q_0}x + w_6 )f
   - xw_6, \\
%%%
  w_1w_3^{q^3} + w_1^{q^3}w_3
  &= (w_3^{q^3}-w_3)w_1 + (w_1^{q^3}-w_1)w_3 - w_1w_3 \\
  &= ( (z^{3q_0})^{q^2}w_1 + (x^{3q_0})^{q^2}w_3 )f^{q^2} \\
  &\qquad + ( (z^{3q_0})^{q}w_1   + (x^{3q_0})^{q}w_3 )f^q
   + (  z^{3q_0}w_1        + x^{3q_0}w_3 )f
   - w_1w_3, \\
%%%
  w_2^{q^3+1}
  &= (w_2^{q^3}-w_2)w_2 + w_2^2 \\
  &= (y^{3q_0})^{q^2}w_2 f^{q^2} + (y^{3q_0})^{q}w_2 f^q + y^{3q_0}w_2 f  + w_2^2 .
\end{align*}
Collecting terms involving common powers of $f$ gives
\[
  A_{-1} + A_0f + A_1f^q + A_2 f^{q^2} ,
\]
where
\[
  A_{-1} = -w_8-xw_6-w_1w_3 + w_2^2
\]
and
\[
  A_i
  = (w_7^{q^i})^{3q_0}
  + (w_4^{q^i})^{3q_0}x
  + (z^{q_i})^{3q_0}w_1
  + (y^{q_i})^{3q_0}w_2
  + (x^{q_i})^{3q_0}w_3
  + w_6
\]
for $i = 0,1,2$.
We claim that $A_{-1} = A_0 = A_1 = 0$.
Indeed, the quadric $A_{-1}$ and each of the three terms in the expression
\begin{multline*}
  A_0
  = (x^{3q_0}w_3 - z^{3q_0}w_1 - w_7^{3q_0} + w_2^{3q_0}) \\
   - (w_4^{3q_0}x + z^{3q_0}w_1 - y^{3q_0}w_2) \\
   - (w_4^{3q_0}x + w_7^{3q_0} + w_2^{3q_0} - w_6)
\end{multline*}
are among the relations listed in Lemma 4.3 of \cite{DE}, so $A_{-1} = 0 = A_0$.
Now using \eqref{wqw}, we obtain
\begin{align*}
  A_0^q - A_1
  &= (w_4^q)^{3q_0}(x^q-x)
  + (w_6^q-w_6) \\
  &\qquad
  + (z^q)^{3q_0}(w_1^q-w_1)
  + (x^q)^{3q_0}(w_3^q-w_3)
  + (y^q)^{3q_0}(w_2^q-w_2) \\
  &= (w_4^q+w_4+z^qx+x^qz + y^{q+1})^{3q_0} (x^q-x) .
\end{align*}
Further simplification reveals that
\[
  B_1 := w_4^q+w_4+z^qx+x^qz + y^{q+1} = 0 ,
\]
and so $A_1 = 0$.

It remains to show that $A_2 = (x^q-x)^{3qq_0+2q+3q_0+1}$.
By \eqref{wqw},
\begin{align*}
  A_1^q - A_2
  &= (w_4^{q^2})^{3q_0}(x^q-x)
  + (w_6^q-w_6) \\
  &\qquad
  + (z^{q^2})^{3q_0}(w_1^q-w_1)
  + (x^{q^2})^{3q_0}(w_3^q-w_3)
  + (y^{q^2})^{3q_0}(w_2^q-w_2) \\
  &= (w_4^{q^2} + w_4 + z^{q^2}x + x^{q^2}z + y^{q^2+1} )^{3q_0}(x^q-x) .
\end{align*}
Thus it suffices to show that
\[
  B_2 := w_4^{q^2} + w_4 + z^{q^2}x + x^{q^2}z + y^{q^2+1} = - (x^q-x)^{q+2q_0+1} .
\]
To do this, we use \eqref{wqw} again, obtaining
\begin{align*}
  B_1^q - B_2
  &= (w_4^q-w_4) + z^{q^2}(x^q-x) + x^{q^2}(z^q-z) + y^{q^2}(y^q-y) \\
  &= (w_2^{q_0} - w_1^{q_0}x^{q_0} + z^{q^2} + x^{q^2}x^{2q_0} + y^{q^2}x^{q_0})(x^q-x) \\
  &= \left[ (xy^{3q_0} - z^{3q_0})^{q_0} - (x^{3q_0+1} - y^{3q_0})^{q_0}x^{q_0} \right.\\
  & \left. \qquad  + z^{q^2} + x^{q^2}x^{2q_0} + y^{q^2}x^{q_0} \right] (x^q-x) \\
  &= \left[ (z^q - z)^q + (y^q - y)^q x^{q_0} + (x^q - x)^q x^{2q_0} \right](x^q-x) \\
  &= (x^{2qq_0} + x^{qq_0+q_0} + x^{2q_0}) (x^q-x)^{q+1} \\
  &= (x^q-x)^{q+2q_0+1} . \qedhere
\end{align*}
\end{proof}

\section{Ray Class Fields}\label{sec RCF}

In this section, we let $X/\FF_q$ denote one the Deligne-Lusztig curves $H$, $S$, or $R$, and let $d = 3,4$, or 6, respectively.
Then $\#X(\FF_q) = q^{d/2}+1$, and $\widetilde X$ is maximal over $\FF_{q^d}$, and the cover $\widetilde X \to X$ is of degree $m = q - \lfloor d/2 \rfloor q_0 +1$.
Moreover, each $\FF_q$-rational point of $X$ is totally ramified in $\widetilde X$ and these are the only points ramified in $\widetilde X \to X$.
The following lemma implies then that every point of $X$ of degree $d$ splits completely in $\widetilde X$ over $\FF_{q^d}$.

\begin{lemma}\label{lemma_Ymax}
   Let $X/\FF_q$ be a curve which is maximal over $\FF_{q^d}$, and let $f\colon Y \to X$ be a tame cover of degree $m > 1$ defined over $\FF_q$.
   Suppose that $f$ is totally ramified at each $\FF_q$-rational point of $X$ and unramified elsewhere.
   Then any two of the following conditions implies the third.
\begin{enumerate}[(i)]
  \item $Y$ is maximal over $\FF_{q^d}$,
  \item $\#X(\FF_q) = q^{d/2}+1$,
  \item $\#(f^{-1}(P) \cap Y(\FF_{q^d})) = m$ for every $P \in X(\FF_{q^d}) \setminus X(\FF_q)$.
\end{enumerate}
\end{lemma}

\begin{proof}
  Let $N_r$ denote the number of $\FF_{q^r}$-rational points of $X$..
  From the Hurwitz genus formula,
  \[
    2g_Y - 2 = m(2g_X-2) + N_1 (m-1) .
  \]
  Also, since $X$ is maximal over $\FF_{q^d}$ we have $N_d = q^d + 1 + 2g_X q^{d/2}$.
  Then
  \begin{align*}
    \#Y(\FF_{q^d})
    &\leq q^d + 1 + 2g_Yq^{d/2} \\
    &= q^d + 1 + q^{d/2}\left[ m(2g_X-2) + N_1(m-1) + 2 \right] \\
    &= N_1 + m ( N_d - N_1 ) + q^{d/2}(m-1)\left[ N_1 - (q^{d/2}+1) \right] ,
  \end{align*}
  If $Y$ is maximal over $\FF_{q^d}$ then equality holds above, and so (ii) is satisfied if and only if (iii) is.
  On the other hand, if both (ii) and (iii) hold, then
  \begin{align*}
    \#Y(\FF_{q^d})
    = N_1 + m (N_d - N_1)
    = q^d + 1 + 2g_Yq^{d/2},
  \end{align*}
  and $Y$ is maximal over $\FF_{q^d}$.
\end{proof}

Define the divisor $\mf m$ as the sum of all $\FF_q$-rational points of $X$ and let $\Sigma$ be the set of the points of $X$ of degree $d$ over $\FF_q$.
Then there is a curve $X_\rcf \to X$ whose function field is the ray class field over $K = \FF_{q^d}(X)$ of conductor $\mf m$ in which each place in $\Sigma$ splits completely.
Since $\widetilde K = \FF_{q^d}(\widetilde X)$ is an abelian extension of $K$ satisfying these ramification and splitting conditions, $\widetilde K$ is contained in $K_\rcf = \FF_{q^d}(X_\rcf)$, and the cover $X_\rcf \to X$ factors through $\widetilde X$.

\begin{theorem}\label{Xrcf max}
  The curve $X_\rcf$ is maximal over $\FF_{q^d}$.
\end{theorem}

\begin{proof}
Write $Y = X_\rcf$.
By Lemma \ref{lemma_Ymax} it suffices to show that $Y \to X$ is totally ramified at each point of $X(\FF_q)$.
Let $k$ be the degree of the cover $Y\to \widetilde X$ and write $N = \#\widetilde X(\FF_q) = q^{d/2}+1$.
For $P \in \widetilde X$, let $e_P$ denote the ramification index of $P$ in $Y$.
Then
\[
  \#Y(\FF_{q^d}) = km \#\Sigma + N - r,
\]
where $r$ is the number of $\FF_q$-rational points of $\widetilde X$ with $e_P < k$.
On the other hand, the Hasse-Weil bound and Riemann-Hurwitz give
\begin{align*}
  \#Y(\FF_{q^d})
  &\leq q^d + 1 + 2g_Y q^{d/2} \\
  &= (q^{d/2}+1)^2 + (2g_Y-2)q^{d/2} \\
  &= N^2 + k q^{d/2}(2g_{\widetilde X} - 2) + q^{d/2} \deg \Diff Y/\widetilde X .
\end{align*}
Now
\[
  \deg \Diff Y/\widetilde X
  = k \sum_{P \in \widetilde X(\FF_q)} \left( 1 - \frac{1}{e_P} \right)
  = N(k-1) - \sum_{e_P > 1} \left( \frac{k}{e_P} - 1 \right) .
\]
Combining all this with the facts $m \#\Sigma = q^{d/2}(q-1) N$ and $2g_{\widetilde X}- 2 = (q-2)N$ and doing some rearranging yields
\begin{equation}\label{Es}
 \sum_{e_P > 1} \left( \frac{k}{e_P} - 1 \right)
 \leq 1 + r q^{-d/2}
 \leq 2 + q^{-d/2} .
\end{equation}
Since $e_P$ divides $k$, each nonzero summand on the left hand side of \eqref{Es} is at least 1, so we conclude that $e_P < k$ for at most two $P \in X(\FF_q)$.
In particular, if $r > 0$ then  either $r=1$ and $e_P = k/3$ for a single $P$, or $r \in \{1,2\}$ and $e_P \geq k/2$ for all $P$.
But either of these cases gives a contradiction in \eqref{Es}.
\end{proof}

\begin{corollary}\label{cyclic}
  The cover $X_\rcf \to X$ is cyclic.
  In particular, the function field $K_\rcf$ is of the form $K( (x^q-x)^{1/mk} )$ for some $k$ dividing $(q^{d/2} + 1)/m$.
\end{corollary}

\begin{proof}
  The first statement follows from the fact that any tame abelian extension $L/K$ which is totally ramified at some place is cyclic.
  Indeed, if not then by replacing $K$ with a larger subfield of $L$ we may assume that $\Gal(L/K) \cong (\ZZ/r\ZZ)^2$ for some $r>1$.
  Then $L$ is a composite of two cyclic Kummer extensions $K_i = K(v_i)$ with $v_i^r = f_i \in K$.
  Since $L/K$ is totally ramified at some place $P$, it follows that $a_i = v_P(f_i)$ is invertible mod $r$ for $i=1,2$ (see \cite{Sticht}, Prop 3.7.3).
  Choose $j$ so that $ja_1 \equiv a_2 \bmod r$.
  Then $v = v_1^j/v_2 \in L$ is a root of $v^r = f_1^j/f_2$, and $K(v)/K$ is unramified at $P$ since $v_P(f_1^j/f_2) = ja_1 - a_2 \equiv 0 \bmod r$.

  Let $J_{\mf m}$ denote the ray class group over $K$ of conductor $\mf m$.
  Then $J_{\mf m}$ fits into an exact sequence
  \[
    1 \to \FF_{q^d}^\times \to \mc O_{\mf m}^\times \to J_{\mf m} \to J \to 1 ,
  \]
  where $J = \Jac(X)(\FF_{q^d})$ and
  \[
    \mc O_{\mf m}^\times
    = \prod_{P \in \mf m} (\mc O_P/\mf m_P)^\times
    \cong (\FF_{q^d}^\times )^{\# X(\FF_{q^2})}
    \cong (\ZZ/(q^d-1)\ZZ)^{q^{d/2}+1} .
  \]
  Furthermore, we have $J \cong (\ZZ/(q^{d/2}+1)\ZZ)^{2g_X}$ since $X$ is maximal over $\FF_{q^d}$.
  Thus $J_{\mf m}$ has exponent dividing $q^d-1$, and so does its quotient $\Gal(K_\rcf/K)$.
  From the discussion above, $L = K( (x^q-x)^{1/(q^d-1)} )$ is the largest abelian extension of $K$ of exponent dividing $q^d-1$ in which each $\FF_q$-rational place of $K$ is totally ramified.
  Since $K_\rcf$ is such an extension, we have $\widetilde K \subset K_\rcf \subset L$, and so $K_\rcf$ is of the form $K((x^q-x)^{1/mk})$ for some $k$ dividing $(q^d-1)/m$.

  This shows that $X_\rcf$ covers the curve $C_{mk}$ given by $u^{mk} = x^q-x$.
  Since $X_\rcf$ is maximal over $\FF_{q^d}$, so is  $C_{mk}$.
  But a theorem of Garcia and Tazafolian \cite{GT} states that a curve of the form $u^r = x^q-x$ may be maximal over $\FF_{q^d}$ only for $r$ dividing $q^{d/2} + 1$.
  Thus $k$ divides $(q^{d/2}+1)/m$, as desired.
\end{proof}

It follows from the proof of Corollary \ref{cyclic} that a sufficient condition for $X_\rcf$ to be equal to $\widetilde X$ is that none of the curves $C_r$ defined by $u^r = x^q-x$ are maximal over $\FF_{q^d}$ for $r$ a proper multiple of $m$.
This is the case when $X=H$, as we show in the following proposition.

\begin{proposition}\label{rmax}
  Let $q = q_0^2$ be a square, and let $r$ be a multiple of $m = q-q_0+1$ which divides $q^{3/2}+1 = (q_0+1)m$.
  Then the curve $C_r$ defined by $u^r = x^q - x$ is maximal over $\FF_{q^3}$ if and only if $r=m$.
\end{proposition}

\begin{proof}
  Let $r = mk$ for some $k$ dividing $q_0+1$.
  Then since $g(C_r) = \frac 12 (r-1)(q-1)$,  $C_r$ is maximal over $\FF_{q^3}$ only if
  \[
    \#C_r(\FF_{q^3}) = q^3 + 1 + (r-1)(q-1)q^{3/2} .
  \]
  Let $\Tr$ denote the field trace from $\FF_{q^3}$ to $\FF_q$.
  Since $r$ divides $q^3-1$, the field $\FF_{q^3}$ contains the $r$th roots of unity.
  Therefore, each $\alpha \in \FF_{q^3}^{\times r}$ has exactly $r$ $r$th roots in $\FF_{q^3}$, and the number of solutions of $\alpha = \beta^q-\beta$ is either $q$ or 0, depending on whether $\Tr(\alpha) = 0$ or not.
  Thus, every element of $\FF_{q^3}^{\times r} \cap \ker \Tr$ contributes $rq$ points to $C_r(\FF_{q^3})$.
  Along with the $q$ points corresponding to $u=0$ and the point at infinity, this gives
  \[
    \#C_r(\FF_{q^3}) = q+1 + rq \cdot \#( \FF_{q^3}^{\times r} \cap \ker \Tr) .
  \]
  Therefore, $C_r$ is maximal over $\FF_{q^3}$ if and only if
  \begin{equation}\label{rTr0}
    \#( \FF_{q^3}^{\times r} \cap \ker \Tr)
    = (q-1)(q_0 + 1/k) .
  \end{equation}
  Since the curve $C_m$ is maximal over $\FF_{q^3}$, we have $\#( \FF_{q^3}^{\times m} \cap \ker \Tr) = (q-1)(q_0+1)$.
  Let $\alpha$ be an element of $\FF_{q^3}^{\times m} \cap \ker \Tr$, so that $\alpha^{(q^3-1)/m} = 1$ and $\Tr(\alpha) = 0$.
  Then
  \begin{align*}
    0
    &= \alpha^{(q^3-1)/m - q^2}\Tr(\alpha) \\
    &= \alpha^{(q^3-1)/m - q^2}(\alpha + \alpha^q + \alpha^{q^2}) \\
    &= \alpha^{(q-1)q_0} + \alpha^{(q-1)(q_0+1)} + 1 ,
  \end{align*}
  so there at most $(q-1)(q_0+1)$ such $\alpha$.
  We conclude that $\FF_{q^3}^{\times m} \cap \ker \Tr$ consists of the roots of the polynomial
  \[
    f(T) = T^{(q-1)(q_0+1)} + T^{(q-1)q_0} + 1 .
  \]
  We claim that the trace zero elements of $\FF_{q^3}^{\times m}$ are evenly distributed among the cosets of the multiplicative subgroup $W = \FF_{q_0^3}^\times \subset \FF_{q^3}^{\times m}$ of index $q_0+1$.
  To see this, first note that the polynomial $f(T)$ admits the factorization
  \[
    f(T)
    = \prod_{\zeta^{q_0+1} = 1} ( T^{q-1} + \zeta T^{q_0-1} + 1 )  .
  \]
  Fix a generator $\beta$ of $\FF_{q^3}^{\times m}$, so that $\zeta = \beta^{q^{3/2}-1}$ is a primitive $(q_0+1)$th root of unity.
  Then we claim that each root of $f(T)$ lying in the coset $\beta^{-i} W$ is a root of $f_i(T) =  T^{q-1} + \zeta^i T^{q_0-1} + 1$.
  For if $\alpha \in \beta^{-i}W$, then $(\alpha \beta^i)^{q_0^3-1} = 1$ and
  \begin{align*}
    \alpha^{(q-1)q_0} f_i(\alpha)
    &= \alpha^{(q-1)q_0}(\alpha^{q-1} + \zeta^i\alpha^{q_0-1} + 1) \\
    &= \alpha^{(q-1)(q_0+1)} + \beta^{(q_0^3-1)i} \alpha^{q_0^3-1} + \alpha^{(q-1)q_0}
    = f(\alpha) .
  \end{align*}
  It follows that $\#(\FF_{q^3}^{r} \cap \ker \Tr) = (q-1)(q_0+1)/k$.
  Comparing with \eqref{rTr0}, we see that $C_r$ is maximal only if $k=1$.
\end{proof}

\begin{corollary}
  The Giulietti-Korchm\'aros curve $\widetilde H$ is equal to $H_\rcf$.
\end{corollary}

The analogue of Proposition \ref{rmax} does not hold in the situation corresponding to $S_\rcf$ and $R_\rcf$.
Indeed, in the Suzuki case $q^{d/2}+1 = q^2+1 = (q+q_0+1)m$, and the curve $u^{q^2+1} = x^q + x$ is covered by the Hermitian curve $u^{q^2+1} = x^{q^2}+x$, hence is maximal over $\FF_{q^4}$.
In the Ree case, there are also proper multiplies $r$ of $m$ such that the curve $u^r = x^q-x$ is maximal over $\FF_{q^6}$.

In any case, Corollary \ref{cyclic} allows one to verify computationally that $X_\rcf = \widetilde X$ for small values of $q$ by checking that $K( (x^q-x)^{1/m\ell})$ is not maximal over $\FF_{q^d}$ for any prime $\ell$ dividing $(q^{d/2}+1)/m$.
Our computations in Magma have shown that $S_\rcf = \widetilde S$ for $q = 2^{2s+1}$ with $1 \leq s \leq 6$, and that $R_\rcf = \widetilde R$ for $q = 27$.
We leave as an open problem the question of whether $X_\rcf = \widetilde X$ in general.
The following bound on the degree of the cover $X_\rcf \to \widetilde X$ follows from Theorem \ref{Xrcf max}.

\begin{corollary}
  The degree $k$ of the cover $X_\rcf \to \widetilde X$ satisfies
  \[
    k \leq \frac{q^{d/2}-3}{q-2} .
  \]
\end{corollary}

\begin{proof}
Since the curve $X_\rcf$ is maximal over $\FF_{q^d}$ its genus is at most $q^{d/2}(q^{d/2}-1)/2$ \cite{Ihara}.
The desired bound follows immediately by combining this with the Hurwitz formula applied to the cover $X_\rcf \to \widetilde X$, and the fact that $2g_{\widetilde X} - 2 = (q-2)(q^{d/2}+1)$.
\end{proof}

\begin{remark*}
This bound is slightly better than the bound $k \leq (q^{d/2}+1)/m$ from Corollary \ref{cyclic}.
When $X$ is the Suzuki or Ree curve, it gives
\begin{align*}
  k &\leq q+2,\\
  k &\leq q^2+2q+4,
\end{align*}
respectively.
Current results on the genus spectrum of maximal curves may be used to reduce these bounds by a factor of 3.
\end{remark*}

\bibliography{rcf_bib}{}

\begin{thebibliography}{10}

\bibitem{Magma}
Wieb Bosma, John Cannon, and Catherine Playoust.
\newblock The {M}agma algebra system. {I}. {T}he user language.
\newblock {\em J. Symbolic Comput.}, 24(3-4):235--265, 1997.
\newblock Computational algebra and number theory (London, 1993).

\bibitem{DL}
P.~Deligne and G.~Lusztig.
\newblock Representations of reductive groups over finite fields.
\newblock {\em Ann. of Math. (2)}, 103(1):103--161, 1976.

\bibitem{DE}
Abdulla Eid and Iwan Duursma.
\newblock Smooth embeddings for the {S}uzuki and {R}ee curves.
\newblock In {\em Algorithmic arithmetic, geometry, and coding theory}, volume
  637 of {\em Contemp. Math.}, pages 251--291. Amer. Math. Soc., Providence,
  RI, 2015.

\bibitem{GSX}
Arnaldo Garcia, Henning Stichtenoth, and Chao-Ping Xing.
\newblock On subfields of the {H}ermitian function field.
\newblock {\em Compositio Math.}, 120(2):137--170, 2000.

\bibitem{GT}
Arnaldo Garcia and Saeed Tafazolian.
\newblock On additive polynomials and certain maximal curves.
\newblock {\em J. Pure Appl. Algebra}, 212(11):2513--2521, 2008.

\bibitem{GK}
Massimo Giulietti and G{\'a}bor Korchm{\'a}ros.
\newblock A new family of maximal curves over a finite field.
\newblock {\em Math. Ann.}, 343(1):229--245, 2009.

\bibitem{Hansen}
Johan~P. Hansen.
\newblock Deligne-{L}usztig varieties and group codes.
\newblock In {\em Coding theory and algebraic geometry ({L}uminy, 1991)},
  volume 1518 of {\em Lecture Notes in Math.}, pages 63--81. Springer, Berlin,
  1992.

\bibitem{HP}
Johan~P. Hansen and Jens~Peter Pedersen.
\newblock Automorphism groups of {R}ee type, {D}eligne-{L}usztig curves and
  function fields.
\newblock {\em J. Reine Angew. Math.}, 440:99--109, 1993.

\bibitem{HS}
Johan~P. Hansen and Henning Stichtenoth.
\newblock Group codes on certain algebraic curves with many rational points.
\newblock {\em Appl. Algebra Engrg. Comm. Comput.}, 1(1):67--77, 1990.

\bibitem{Henn}
Hans-Wolfgang Henn.
\newblock Funktionenk\"orper mit grosser {A}utomorphismengruppe.
\newblock {\em J. Reine Angew. Math.}, 302:96--115, 1978.

\bibitem{HKT}
J.~W.~P. Hirschfeld, G.~Korchm{\'a}ros, and F.~Torres.
\newblock {\em Algebraic curves over a finite field}.
\newblock Princeton Series in Applied Mathematics. Princeton University Press,
  Princeton, NJ, 2008.

\bibitem{Huppert}
Bertram Huppert and Norman Blackburn.
\newblock {\em Finite groups. {III}}, volume 243 of {\em Grundlehren der
  Mathematischen Wissenschaften [Fundamental Principles of Mathematical
  Sciences]}.
\newblock Springer-Verlag, Berlin-New York, 1982.

\bibitem{Ihara}
Yasutaka Ihara.
\newblock Some remarks on the number of rational points of algebraic curves
  over finite fields.
\newblock {\em J. Fac. Sci. Univ. Tokyo Sect. IA Math.}, 28(3):721--724 (1982),
  1981.

\bibitem{KT}
G{\'a}bor Korchm{\'a}ros and Fernando Torres.
\newblock Embedding of a maximal curve in a {H}ermitian variety.
\newblock {\em Compositio Math.}, 128(1):95--113, 2001.

\bibitem{Lachaud}
Gilles Lachaud.
\newblock Sommes d'{E}isenstein et nombre de points de certaines courbes
  alg\'ebriques sur les corps finis.
\newblock {\em C. R. Acad. Sci. Paris S\'er. I Math.}, 305(16):729--732, 1987.

\bibitem{Lauter}
Kristin Lauter.
\newblock Deligne-{L}usztig curves as ray class fields.
\newblock {\em Manuscripta Math.}, 98(1):87--96, 1999.

\bibitem{MZ}
M.~Montanucci and G.~Zini.
\newblock Some {R}ee and {S}uzuki curves are not {G}alois covered by the
  {H}ermitian curve.
\newblock 2016.
\newblock Pre-print, \url{arXiv:1603.06706}.

\bibitem{Ped}
Jens~Peter Pedersen.
\newblock A function field related to the {R}ee group.
\newblock In {\em Coding theory and algebraic geometry ({L}uminy, 1991)},
  volume 1518 of {\em Lecture Notes in Math.}, pages 122--131. Springer,
  Berlin, 1992.

\bibitem{Sticht}
Henning Stichtenoth.
\newblock {\em Algebraic Function Fields and Codes}.
\newblock Springer Publishing Company, Incorporated, 2nd edition, 2008.

\end{thebibliography}
\bibliographystyle{plain}

\end{document}